\documentclass[12pt,oneside,english]{amsart}
\usepackage[latin1]{inputenc}
\usepackage{amssymb}

\makeatletter
\newtheorem{theorem}{Theorem}
\newtheorem{lemma}{Lemma}

\newtheorem{proposition}{Proposition}
\newtheorem{remark}{Remark}

\usepackage{babel}

\makeatother
\begin{document}
\baselineskip=17pt
\title[Set of all densities ]{Set of all densities of exponentially S-numbers}

\author{Vladimir Shevelev}
\address{Department of Mathematics \\Ben-Gurion University of the
 Negev\\Beer-Sheva 84105, Israel. e-mail:shevelev@bgu.ac.il}


\begin{abstract}
Let $\mathbf{G}$  be the set of all finite or infinite increasing sequences
of positive integers beginning with 1. For a sequence $S=\{s(n)\}, n\geq1,$ from
$\mathbf{G}$ a positive number $N$ is called an exponentially
$S$-number $(N\in E(S)),$ if all exponents in its prime power factorization are in
$S.$ The author \cite{2} proved that, for every sequence $S\in \mathbf{G},$
the sequence of exponentially $S$-numbers has a density $h=h(E(S))\in [\frac{6}{\pi^2}, 1].$
In this note we study the set $\{h(E(S)\}$ of all such densities.
\end{abstract}

\maketitle

\section{Introduction}
Let $\mathbf{G}$ be the set of all finite or infinite increasing sequences
of positive integers beginning with 1. For a sequence $S=\{s(n)\}, n\geq1,$ from
$\mathbf{G},$ a positive number $N$ is called an exponentially
$S$-number $(N\in E(S)),$ if all exponents in its prime power factorization are in
$S.$ The author \cite{2} proved that, for every sequence $S\in \mathbf{G},$
the sequence of exponentially $S$-numbers has a density
$h=h(E(S))\in [\frac{6}{\pi^2}, 1].$
More exactly, the following theorem was proved in \cite{2}:
\begin{theorem}\label{t1}
For every sequence $S\in \mathbf{G}$ the sequence of exponentially $S$-numbers has
a density $h=h(E(S))$ such that
\begin{equation}\label{1}
\sum_{i\leq x,\enskip i\in E(S)} 1 = h(E(S))x+O(\sqrt{x}\log x e^{c\frac{\sqrt{\log x}}
{\log \log x}}),
\end{equation}
with $c=4\sqrt{\frac{2.4}{\log 2}}=7.443083...$ and
\begin{equation}\label{2}
h(E(S))=\prod_{p}\left(1+\sum_{i\geq2}\frac{u(i)-u(i-1)}{p^i}\right),
\end{equation}
where $u(n)$ is the characteristic function of sequence $S:\enskip u(n)=1,$ if\enskip
 $n\in S$ and $u(n)=0$ otherwise.
\end{theorem}

\indent Note that Sloane's Online Encyclopedia of Integer Sequences \cite{3}
contains some sequences of exponentially $S$-numbers, $S\in \mathbf{G}.$
For example, A005117 $(S=\{1\}),$ A004709 $(S=\{1,2\}),$ A268335 $(S=A005408),$ A138302
$(S=\{2^n\}|_{n\geq0}),$ A197680 $(S=\{n^2\}|_{n\geq1}),$ A115063 $(S=\{F_n\}|_{n\geq2}),$ 
A209061 $(S=A005117),$
etc.\newline
\indent Everywhere below we write $\{h(E(S))\},$ understanding
$\{h(E(S))\}|_{S\in \mathbf{G}}.$
In \cite{2} (Section 6) the author posed the question: is the set $\{h(E(S))\}$
dense in the interval $[\frac{6}{\pi^2}, 1]?$
Berend \cite{1} gave a negative answer by finding a gap in the set
$\{h(E(S))\}$ in the interval
\begin{equation}\label{3}
\left(\prod_{p}(1-\frac{p-1}{p^3}),\enskip \prod_{p}(1-\frac{1}{p^3}) \right)\subset
[\frac{6}{\pi^2}, 1].
\end{equation}
Berend's idea consists of the partition of $\mathbf{G}$ into two subsets - of those
sequences which contain 2 and those that do not contain 2 - and applying Theorem
\ref{t1}. In our study of the set $\{h(E(S)\}$ we use this idea.

\section{Cardinality}
\begin{lemma}\label{L1}
$\mathbf{G}$ is uncountable.
\end{lemma}
\begin{proof}
Trivially $\mathbf{G}$ is equivalent to the set of all subsets of $\{2,3,4,...\}.$
\end{proof}
\begin{lemma}\label{L2}
For every two distinct $A, B\in \mathbf{G},$ we have $h(E(A))\neq h(E(B)).$
\end{lemma}
\begin{proof}
Let $A=\{a(i)\}|_{i\geq1},\enskip B=\{b(i)\}|_{i\geq1}.$ Let $n\geq1$ be maximal index
such that $a(i)=b(i),\enskip i=1,...,n,$ while $a(n+1)\neq b(n+1).$
Note that, if $A_n=\{a(1),...,a(n)\}, \enskip A^*=\{a(1),...,a(n), a(n+1), a(n+1)+1,a(n+1)+2,...\},$ then
\begin{equation}\label{4}
h(E(A_{n+1}))\leq h(E(A))\leq h(E(A^*))
\end{equation}
and analogously for sequence $B.$
\newline
Distinguish four cases:
$$(i)\enskip a(n+1)=a(n)+1,\enskip b(n+1)\geq a(n)+2;$$
$$(ii)\enskip for\enskip k\geq2,\enskip a(n+1)\geq a(n)+k, \enskip b(n+1)=a(n)+1;$$
$$(iii) \enskip for \enskip k\geq3,\enskip a(n+1)=a(n)+k,\enskip a(n)+2\leq b(n+1)
\leq a(n)+k-1;$$
$$(iv) \enskip for \enskip k\geq2,\enskip a(n+1)=a(n)+k,\enskip b(n+1)\geq a(n)+k+1.$$
(i) By (\ref{2}) and (\ref{4}), we have
\begin{equation}\label{5}
h(E(A))\geq \prod_{p}\left(1+\sum_{i=2}^{a(n)}\frac{u(i)-u(i-1)}{p^i}\right),
\end{equation}
where $u(n)$ is the characteristic function of $A.$ Since here $u(a(n+1))-
u(a(n+1)-1)=0,$ then in the right hand side we sum up to $a(n).$
On the other hand,
\begin{equation}\label{6}
h(E(B^*))\leq \prod_{p}\left(1+\sum_{i=2}^{a(n)}\frac{u(i)-u(i-1)}{p^i}-
\frac{1}{p^{a(n)+1}}+\frac{1}{p^{a(n)+2}}\right).
\end{equation}
By (\ref{5})-(\ref{6}), $h(E(B))<h(E(A)).$
\newpage
(ii) Symmetrically to (i), we have
\begin{equation}\label{7}
h(E(B))\geq \prod_{p}\left(1+\sum_{i=2}^{a(n)}\frac{u(i)-u(i-1)}{p^i}\right).
\end{equation}
On the other hand,
\begin{equation}\label{8}
h(E(A^*))\leq \prod_{p}\left(1+\sum_{i=2}^{a(n)}\frac{u(i)-u(i-1)}{p^i}-
\frac{1}{p^{a(n)+1}}+\frac{1}{p^{a(n)+2}}\right).
\end{equation}
So, $h(E(A))<h(E(B)).$\newline
(iii) Again, by (\ref{2}) and (\ref{4}), we have
\begin{equation}\label{9}
h(E(B))\geq \prod_{p}\left(1+\sum_{i=2}^{a(n)}\frac{u(i)-u(i-1)}{p^i}-
\frac{1}{p^{a(n)+1}}+\frac{1}{p^{a(n)+k-1}}\right),
\end{equation}
while
\begin{equation}\label{10}
h(E(A^*))\leq \prod_{p}\left(1+\sum_{i=2}^{a(n)}\frac{u(i)-u(i-1)}{p^i}-
\frac{1}{p^{a(n)+1}}+\frac{1}{p^{a(n)+k}}\right).
\end{equation}
Hence, $h(E(A))<h(E(B)).$\newline
(iv) Symmetrically,
\begin{equation}\label{11}
h(E(B^*))\leq \prod_{p}\left(1+\sum_{i=2}^{a(n)}\frac{u(i)-u(i-1)}{p^i}-
\frac{1}{p^{a(n)+1}}+\frac{1}{p^{a(n)+k+1}}\right),
\end{equation}
while
\begin{equation}\label{12}
h(E(A))\geq \prod_{p}\left(1+\sum_{i=2}^{a(n)}\frac{u(i)-u(i-1)}{p^i}-
\frac{1}{p^{a(n)+1}}+\frac{1}{p^{a(n)+k}}-\frac{1}{p^{a(n)+k+1}}\right)
\end{equation}
and since $\frac{2}{p^{a(n)+k+1}}\leq \frac{1}{p^{a(n)+k}},$ where the equality holds
only in case $p=2,$ then $h(E(A))>h(E(B)).$
\end{proof}
Lemmas \ref{L1} and \ref{L2} directly imply
\begin{theorem}\label{t2}
The set $\{h(E(S)\}|_{S\in \mathbf{G}}$ is uncountable.
\end{theorem}
Denote by $\mathbf{G}(F)$ the subset of the finite sequences from $\mathbf{G}.$
Since the set of all finite subsets of a countable set is countable, then
$\mathbf{G}(F)$ is countable and then the set
$\{h(E(S)\}|_{S\in \mathbf{G}(F)}$ is also countable.

\section{Perfectness}
\begin{lemma}\label{L3}
Every point of the set $h(E(S))$ is an accumulation point.
\end{lemma}
\begin{proof}
Distinguish two cases: a) $S$ is finite set; b) $S$ is infinite set.
\newpage
a) Let $S=\{s(1),...,s(k)\}\in \mathbf{G}(F).$ Let $n\geq s(k)+2.$ Denote by
$S_n$ the
sequence $S_n=\{s(1),...,s(k),n\}.$ Then, by (\ref{2}),
\begin{equation}\label{13}
h(E(S_n))-h(E(S))=
\end{equation}
$$\prod_{p}\left(1+\sum_{i=2}^{s(k)}\frac{u(i)-u(i-1)}{p^i}-\frac{1}{p^{s(k)+1}}
+\frac{1}{p^n}\right) - $$ $$\prod_{p}\left(1+\sum_{i=2}^{s(k)}\frac{u(i)-u(i-1)}{p^i}-
\frac{1}{p^{s(k)+1}}\right).$$
For the first product $\prod_{p}(n),$
$$\prod_{p}(n)=\exp\left(\sum_{p}\log\left(1+\sum_{i=2}^{s(k)}\frac{u(i)-u(i-1)}{p^i}-\frac{1}{p^{s(k)+1}}
+\frac{1}{p^n}\right)\right),$$
the series over primes converges uniformly since
$$\sum_{p}\sum_{i\geq2}\frac{|u(i)-u(i-1)|}{p^i}\leq \sum_{p}\sum_{i\geq2}
\frac{1}{p^i}=\sum_{p}\frac{1}{(p-1)p}.$$
Therefore,
$\lim_{n\rightarrow\infty}(\prod_{p}(n))=\prod_{p}(\lim_{n\rightarrow\infty}(...)) $
which coincides with the second product. So $\lim_{n\rightarrow\infty}
 h(E(S_n))=h(E(S)).$\newline
 b) Let $S=\{s(1),...,s(k),...\}\in \mathbf{G}$ be infinite sequence. Let  $S_n=\{s(1),...,s(n)\}$ be the $n$-partial sequence of $S.$ \enskip In the same way, taking
 into account the uniform convergence of the product for density of $S_n,$ we find that
 $\lim_{n\rightarrow\infty} h(E(S_n))=h(E(S)).$
\end{proof}
\begin{theorem}\label{t3}
The set $\{h(E(S))\}$ is a perfect set.
\end{theorem}
A proof we give in Section 5.
\section{Gaps}
Let us show that, for every finite $S\in\mathbf{G},$ with the exception of
$S=\{1\},$ there exists an $\varepsilon>0$ such that the image of $h$ is disjoint
from the interval $(h(E(S))-\varepsilon, h(E(S)).$

We need a lemma.
\begin{lemma}\label{L4}
Let $A, B\in \mathbf{G}$ be distinct sequences. Let $s^*=s^*(A,B)$ be the smallest
number which is a term of one of them, but not in another. If, say, $s^*\in A,$
then $h(E(A))>h(E(B)).$
\end{lemma}
\begin{proof}
In fact, the lemma is a corollary of the proof of Lemma \ref{L2}. Comparing with the proof of Lemma
 \ref{L2}, we have $s^*(A,B)=n+1.$ We see that in all four
cases in the proof of Lemma \ref{L2}, the statement of Lemma \ref{L4} is confirmed.
\end{proof}
\newpage
\begin{proposition}\label{P1}
Let $S_1=\{s(1),...,s(k)\}\in \mathbf{G}(F), \enskip k\geq2,$ and
\enskip$S_2=\{s(1),...,s(k-1),s(k)+1,s(k)+2,...\}.$ Then the interval
\begin{equation}\label{14}
 (h(E(S_2)), \enskip h(E(S_1)))
\end{equation}
is a gap in the set $\{h(E(S)): S\in \mathbf{G}\}.$
\end{proposition}
\begin{proof}
Consider other than $S_1, S_2$ any sequence $S\in \mathbf{G}$ which contains
$s^*(S_1,S).$ By Lemma \ref{L4}, $h(E(S))>h(E(S_1)).$ So, $h(E(S))$ is not in interval
(\ref{14}). Now consider other than $S_1, S_2$ any sequence $S\in \mathbf{G}$ which
does not contain $s^*(S_1,S).$ Then $S_2$ contains $s^*(S,S_2).$ Indeed, 1) $S$ cannot
contain all terms $s(1),...,s(k)$ (since $S$ differs from $S_1,$ it should contain
additional terms, the smallest of which is $s^*(S,S_1)\in S$ that contradicts
the condition); 2) if $i,\enskip 1\leq i\leq k,$ is the smallest for which $S$ misses
$s(i),$ then, by the
condition, all terms of $S$ are more than than $s(i).$ So
$s^*(S,S_2)=s(i)\in S_2,$ if $i<k,$ while, if i=k, since $S$ differs from
$S_2,\enskip s^*(S,S_2)=s(k)+j\in S_2,$ where $j$ is the smallest for which $s_k+j$ is not
in $S.$ Hence, by Lemma \ref{L4}, $h(E(S_2))>h(E(S))$ and again $h(E(S))$ is not in
interval (\ref{14}).
\end{proof}
\begin{lemma}\label{L5}
Every gap in $\{h(E(S))\}$ has the form described in Proposition
$\ref{P1}.$
\end{lemma}
\begin{proof}
Indeed, the gap (\ref{14}) is in a right neighborhood of $h(E(S_2)).$
Let a sequence $S\in \mathbf{G}$ do not contain any infinite
set of positive integers $K.$ Adding to $S$ $k\in K,$ which goes to infinity,
we obtain set $S_k$ such that $h(E(S_k))>h(E(S))$ and $h(E(S_k))\rightarrow h(E(S)).$
So, in a right neighborhood of $h(E(S))$ cannot be a gap of $\{h(E(S))\}.$
In opposite case, when $S\in \mathbf{G}$ does not contain only a finite set of
positive integers, in a right neighborhood of $h(E(S))$ a gap of $\{h(E(S))\}$
is possible, but in this case $S$ has the form of $S_2$ in Proposition \ref{P1}.
Also, if $S\in \mathbf{G}$ is infinite, then in a left neighborhood of $h(E(S))$
cannot be a gap of $\{h(E(S))\},$ since $h(E(S))$ is a limiting point of
$\{h(E(S_n))\},$ where $S_n$ is the $n$-partial sequence of $S.$
\end{proof}
It is easy to see that, for distinct sequences $S_1,$ the gaps $(\ref{14})$ are
disjoint.

From Propositions \ref{P1} and Lemma \ref{L5} we have the statement:
\begin{theorem}\label{t4}
The set $\{h(E(S))\}$ has countably many gaps.
\end{theorem}

\section{Proof of Theorem 3}
\begin{proof}
By Lemma \ref{L3}, the set $\{h(E(S))\}$ does not contain isolated points.
For a set $A\subseteq [\frac{6}{\pi^2}, 1],$ let $\overline{A}$ be $[\frac{6}{\pi^2}, 1]
\backslash A.$ Let, further, $\{g\}$ be the set of all gaps of $\{h(E(S))\}.$ Then we have
\newpage
$$\{h(E(S))\}=\overline {\bigcup g}=\bigcap\overline{g}. $$
Since a gap $g$ is an open interval, then $\overline{g}$ is a closed set.
But arbitrary intersections of closed sets are closed. Thus the set $\{h(E(S))\}$
is closed without isolated points. So it is a perfect set.
\end{proof}

\section{Conclusion}
Thus, by Theorems \ref{t2}-\ref{t4}, the set $\{h(E(S))\}$ is a perfect
set with a countable set of gaps which associate with some left-sided neighborhoods
of the densities of all exponentially finite $S$-sequences,
$S\in \mathbf{G},$ except for $S=\{1\}.$ It is natural to conjecture that the sum of lengths
of all gaps equals the length of the whole interval $[\frac{6}{\pi^2}, 1],$ or, the
same, that the set $\{h(E(S))\}$ has zero measure. This important question
we remain open.
\begin{remark} Possible to solve this problem could help a remark that the deleting
in $(\ref{2})$ 0's (when $u_i=u_{i-1})$ we obtain an alternative sequence of $-1,1.$
\end{remark}

\section{Acknowledgement}
The author is grateful to Daniel Berend for very useful discussions.

\end{document}